\documentclass{article}%
\usepackage{amsmath}
\usepackage{amsfonts}
\usepackage{amssymb}
\usepackage{graphicx}%
\providecommand{\U}[1]{\protect \rule{.1in}{.1in}}
\newtheorem{theorem}{Theorem}

\newtheorem{definition}[theorem]{Definition}

\newtheorem{lemma}[theorem]{Lemma}

\newtheorem{proposition}[theorem]{Proposition}
\newtheorem{remark}[theorem]{Remark}

\newenvironment{proof}[1][Proof]{\noindent \textbf{#1.} }{\  \rule{0.5em}{0.5em}}
\begin{document}

\title{Some Estimates for Martingale Representation under $G$-Expectation}

\author{Ying HU \thanks{IRMAR, Universit\'{e}
de Rennes I, France},  Shige PENG\thanks{Institute of Mathematics,
Institute of Finance, Shandong University, 250100, Jinan, China,
School of Mathematics, Fudan University, peng@sdu.edu.cn, The author
thanks the partial support from The National Basic Research Program
of China (973 Program) grant No. 2007CB814900 (Financial Risk) .}}
\maketitle

\begin{abstract}
We provides some useful estimates for solving martingale representation
problem under $G$-expectations. We also study the corresponding conditions for
the existence and uniqueness.

\end{abstract}

\section{Introduction}

Many important progresses in the domain of $G$-expectation and related
$G$-Brownian motion have been made in recent years since the introduction of
this theory. But some very important questions still remain open. An
interesting and challenging one is the so called $G$-martingale representation
problem proposed in \cite[Peng2007]{Peng2007} of the following form:
\begin{equation}
M_{t}=\mathbb{E}_{G}[X|\Omega_{t}]=\mathbb{E}_{G}[X]+\int_{0}^{t}z_{s}%
dB_{s}+\int_{0}^{t}\eta_{s}d\left \langle B\right \rangle _{s}-\int_{0}%
^{t}2G(\eta_{s})ds, \label{0.1}%
\end{equation}
for some given element $X\in L_{G}^{2}(\Omega_{T})$. As Peng pointed out in
many of his lectures and discussions, this formulation permits us to treat the
following version of BSDE driven by $G$-Brownian motion of the form%
\[
y_{t}=X+\int_{t}^{T}f(s,y_{s},z_{s},\eta_{s})ds-\int_{t}^{T}z_{s}dB_{s}%
-\int_{t}^{T}\eta_{s}d\left \langle B\right \rangle _{s}+\int_{t}^{T}2G(\eta
_{s})ds.
\]
\cite[Peng2007, Peng2010]{Peng2007, Peng2010} had only treated the
above form of $G$-martingale for the situation where $X\in
L_{ip}(\Omega_{T})$ and $G$ is non-degenerate.

The first step towards a proof to (\ref{0.1}) for a more general $X\in
L_{G}^{2}(\Omega_{T})$ was given by \cite{XuJing} for the case where $X$
satisfies $\mathbb{E}_{G}[X]=-\mathbb{E}_{G}[-X]$. In this case the
$G$-martingale $M_{t}=\mathbb{E}_{G}[X|\Omega_{t}],t\geq0$, is symmetric,
namely $-M$ is also a $G$-martingale. They obtained a representation of the
form $X=\mathbb{E}_{G}[X]+\int_{0}^{T}z_{s}dB_{s}$, corresponding to the case
$\eta \equiv0$.

Observe that in the formulation of (\ref{0.1}) $M_{t}=M_{t}^{s}-A_{t}$, where
$M^{s}=\mathbb{E}_{G}[X]+\int_{0}^{t}z_{s}dB_{s}$ is a symmetric martingale
and
\[
A_{t}=\int_{0}^{t}2G(\eta_{s})ds-\int_{0}^{t}\eta_{s}d\left \langle
B\right \rangle _{s},\  \ t\in \lbrack0,T]
\]
is a nondecreasing process with $A_{0}=0$ such that $-A_{t}$ is a
$G$-martingale. A very interesting problem is whether $M$ has a unique
decomposition $M^{s}-A$. An important progress of this problem was obtained by
Soner, Touzi and Zhang \cite{STZ}. They have proved that, under the condition%
\[
\left \Vert X\right \Vert _{\mathbb{L}_{\mathcal{P}}^{2}}:=\mathbb{E}_{G}%
[\sup_{t\in \lbrack0,T]}|M_{t}|^{2}]]<\infty
\]
there exists a unique decomposition $M=M^{s}-A$ such that
\[
\mathbb{E}_{G}[A_{T}^{2}]+\mathbb{E}_{G}[\int_{0}^{T}|z_{s}|^{2}ds]\leq
C^{\ast}\left \Vert X\right \Vert _{\mathbb{L}_{\mathcal{P}}^{2}},
\]
where $C^{\ast}$ is a universal constant. More recently, \cite[Song,2010]%
{Song} has significantly improved their result by only assuming that $X\in
L_{G}^{p}(\Omega_{T})$ for $p>1$. His result will be used in this paper (see
the next section).

In this paper we give an a priori estimate of $\eta$ in the representation of
(\ref{0.1}): if for a fixed $\varepsilon>0$, $G_{\varepsilon}%
(a)=G(a)-\varepsilon/2|a|$ is a sublinear function of $a$, then we have
\[
\mathbb{E}_{G_{\varepsilon}}\int_{0}^{T}|\eta_{t}|dt\leq \varepsilon
^{-1}(\mathbb{E}_{G}[X]+\mathbb{E}_{G_{\varepsilon}}[-X]).
\]
This estimate indicates clearly that the norm of $\eta$ can be dominated by
$\mathbb{E}_{G}[X]+\mathbb{E}_{G_{\varepsilon}}[-X]$ and thus provides a new
proof of \cite{Peng2007,Peng2010} for the existence of the representation in
the case $X\in L_{ip}(\Omega_{T})$. We will also give a proof of uniqueness of
$\eta$. We then show that, if the increasing part $A$ of $M$ satisfies a type
of bounded variation condition, then we can also prove the existence of the representation.

This paper is organized as follows: after given some basic settings in the
next section, we give the a priori estimate and then a proof of the
representation for the case where $X\in L_{ip}(\Omega_{T})$, in Section 3. In
Section 4 we study the uniqueness of representation theorem for $X\in
L_{G}^{p}(\Omega_{T})$. In Section 5 we study the existence of the
representation of $G$-martingales.

\section{Preliminaries}

We present some preliminaries in the theory of sublinear expectations and the
related $G$-Brownian motions. More details can be found in Peng
\cite{Peng2006a}, \cite{Peng2006b} and \cite{Peng2007}.

\begin{definition}
{\label{Def-1} { Let }}$\Omega$ be a given set and let $\mathcal{H}$ be a
linear space of real valued functions defined on $\Omega$ with $c\in
\mathcal{H}$ for all constants $c$, and $|X|\in \mathcal{H}$, if $X\in
\mathcal{H}$. $\mathcal{H}$ is considered as the space of our
\textquotedblleft random variables\textquotedblright. {{A \textbf{sublinear
expectation }$\mathbb{\hat{E}}$ on $\mathcal{H}$ is a functional
$\mathbb{\hat{E}}:\mathcal{H}\mapsto \mathbb{R}$ satisfying the following
properties: for all $X,Y\in \mathcal{H}$, we have\newline \  \  \newline%
\textbf{(a) Monotonicity:} \  \  \  \  \  \  \  \  \  \  \  \  \ If $X\geq Y$ then
$\mathbb{\hat{E}}[X]\geq \mathbb{\hat{E}}[Y].$\newline \textbf{(b) Constant
preserving: \  \ }\  \ $\mathbb{\hat{E}}[c]=c$.\newline \textbf{(c)}
\textbf{Sub-additivity: \  \  \  \ }}}\  \  \  \  \  \  \  \ $\mathbb{\hat{E}%
}[X]-\mathbb{\hat{E}}[Y]\leq \mathbb{\hat{E}}[X-Y].$\newline{{\textbf{(d)
Positive homogeneity: } \ $\mathbb{\hat{E}}[\lambda X]=\lambda \mathbb{\hat{E}%
}[X]$,$\  \  \forall \lambda \geq0$.\newline}}\newline The triple $(\Omega
,\mathcal{H},\mathbb{\hat{E}})$ is called a sublinear expectation space.
$X\in \mathcal{H}$ is called a random variable in $(\Omega,\mathcal{H})$. We
often call $Y=(Y_{1},\cdots,Y_{d})$, $Y_{i}\in \mathcal{H}$ a $d$-dimensional
random vector in $(\Omega,\mathcal{H})$. Let us consider a space of random
variables $\mathcal{H}$ satisfying: if $X_{i}\in \mathcal{H}$, $i=1,\cdots,d$,
then%
\[
\varphi(X_{1},\cdots,X_{d})\in \mathcal{H}\text{,\  \ for all }\varphi \in
C_{b,Lip}(\mathbb{R}^{d}),
\]
where $C_{b,Lip}(\mathbb{R}^{d})$ is the space of all bounded and Lipschitz
continuous functions on $\mathbb{R}^{d}$. An $m$-dimensional random vector
$X=(X_{1},\cdots,X_{m})$ is said to be independent of another $n$-dimensional
random vector $Y=(Y_{1},\cdots,Y_{n})$ if
\[
\mathbb{\hat{E}}[\varphi(X,Y)]=\mathbb{\hat{E}}[\mathbb{\hat{E}}%
[\varphi(X,y)]_{y=Y}],\  \  \text{for }\varphi \in C_{b,Lip}(\mathbb{R}^{m}%
\times \mathbb{R}^{n}).
\]
Let $X_{1}$ and $X_{2}$ be two $n$--dimensional random vectors defined
respectively in {sublinear expectation spaces }$(\Omega_{1},\mathcal{H}%
_{1},\mathbb{\hat{E}}_{1})${ and }$(\Omega_{2},\mathcal{H}_{2},\mathbb{\hat
{E}}_{2})$. They are called identically distributed, denoted by $X_{1}\sim
X_{2}$, if
\[
\mathbb{\hat{E}}_{1}[\varphi(X_{1})]=\mathbb{\hat{E}}_{2}[\varphi
(X_{2})],\  \  \  \forall \varphi \in C_{b.Lip}(\mathbb{R}^{n}).
\]
If $X$, $\bar{X}$ are two $m$-dimensional random vectors in $(\Omega
,\mathcal{H},\mathbb{\hat{E}})$ and $\bar{X}$ is identically distributed with
$X$ and independent of $X$, then $\bar{X}$ is said to be an independent copy
of $X$.
\end{definition}

\begin{definition}
(\textbf{$G$-normal distribution}) \label{Def-Gnormal} A $d$-dimensional
random vector $X=(X_{1},\cdots,X_{d})$ in a sublinear expectation space
$(\Omega,\mathcal{H},\mathbb{\hat{E}})$ is called $G$-normal distributed if
for each $a\,$, $b\geq0$ we have
\begin{equation}
aX+b\bar{X}\sim \sqrt{a^{2}+b^{2}}X,\  \label{srt-a2b2}%
\end{equation}
where $\bar{X}$ is an independent copy of $X$. Here the letter $G$ denotes the
function
\[
G(A):=\frac{1}{2}\mathbb{\hat{E}}[(AX,X)]:\mathbb{S}_{d}\mapsto \mathbb{R}.
\]
It is also proved in Peng \cite{Peng2006b, Peng2007} that, for each
$\mathbf{a}\in \mathbb{R}^{d}$ and $p\in \lbrack1,\infty)$%
\[
\mathbb{\hat{E}}[|\left(  \mathbf{a},X\right)  |^{p}]=\frac{1}{\sqrt
{2\pi \sigma_{\mathbf{aa}^{T}}^{2}}}\int_{-\infty}^{\infty}|x|^{p}\exp \left(
\frac{-x^{2}}{2\sigma_{\mathbf{aa}^{T}}^{2}}\right)  dx,
\]
where $\sigma_{\mathbf{aa}^{T}}^{2}=2G(\mathbf{aa}^{T})$.
\end{definition}

\begin{definition}
A $d$-dimensional stochastic process $\xi_{t}(\omega)=(\xi_{t}^{1},\cdots
,\xi_{t}^{d})(\omega)$ defined in a sublinear expectation space $(\Omega
,\mathcal{H},\mathbb{\hat{E}})$ is a family of $d$-dimensional random vectors
$\xi_{t}$ parameterized by $t\in \lbrack0,\infty)$ such \ that $\xi_{t}^{i}%
\in \mathcal{H}$, for each $i=1,\cdots,d$ and $t\in \lbrack0,\infty)$.
\end{definition}

The most typical stochastic process in a sublinear expectation space is the
so-called $G$-Brownian motion.

\begin{definition}
\label{Def-3}(\textbf{\cite{Peng2006a} and \cite{Peng2007}}) Let
$G:\mathbb{S}_{d}\mapsto \mathbb{R}$ be a given monotonic and sublinear
function. A process $\{B_{t}(\omega)\}_{t\geq0}$ in a sublinear expectation
space $(\Omega,\mathcal{H},\mathbb{\hat{E}})$ is called a $G$%
\textbf{--Brownian motion} if for each $n\in \mathbb{N}$ and $0\leq
t_{1},\cdots,t_{n}<\infty$,$\ B_{t_{1}},\cdots,B_{t_{n}}\in \mathcal{H}$ and
the following properties are satisfied: \newline \textsl{(i)} $B_{0}(\omega
)=0$;\newline \textsl{(ii)} For each $t,s\geq0$, the increment $B_{t+s}-B_{t}$
is independent of $(B_{t_{1}},B_{t_{2}},\cdots,B_{t_{n}})$, for each
$n\in \mathbb{N}$ and $0\leq t_{1}\leq \cdots \leq t_{n}\leq t$; \newline%
\textsl{(iii)} $B_{t+s}-B_{s}\sim B_{t}$, for $s,t\geq0$ and $\mathbb{\hat{E}%
}[|B_{t}|^{3}]/t\rightarrow0$, as $t\rightarrow0$. \newline \textsl{(vi)
}$\mathbb{E}[B_{t}]=-\mathbb{E}[-B_{t}]=0$, for $t\geq0$.
\end{definition}

It was proved that, for each $t>0$, $B_{t}/\sqrt{t}$ is $G$-normal distributed
with $G(A)=\frac{1}{2}\mathbb{\hat{E}}[\left \langle AB_{1},B_{1}\right \rangle
]$. In many cases $B$ is also called $G$-Brownian motion when it only
satisfies (i)-(iii), and a $G$-Brownian motion satisfying (i)-(iv) is called
symmetric $G$-Brownian motion. In this paper we only discuss symmetric
$G$-Brownian motions.

We denote:

$\Omega=C_{0}^{d}(\mathbb{R}^{+})$ the space of all $\mathbb{R}^{d}$-valued
continuous functions $(\omega_{t})_{t\in \mathbb{R}^{+}}$, with $\omega_{0}=0$,
equipped with the distance
\[
\rho(\omega^{1},\omega^{2}):=\sum_{i=1}^{\infty}2^{-i}[(\max_{t\in \lbrack
0,i]}|\omega_{t}^{1}-\omega_{t}^{2}|)\wedge1].
\]
We denote by $\mathcal{B}(\Omega)$ the Borel $\sigma$-algebra of $\Omega$ and
by $\mathcal{M}$ the collection of all probability measure on $(\Omega
,\mathcal{B}(\Omega))$.

We also denote, for each $t\in \lbrack0,\infty)$:

\begin{itemize}
\item $\Omega_{t}:=\{ \omega_{\cdot \wedge t}:\omega \in \Omega \}$,

\item $\mathcal{F}_{t}:=\mathcal{B}(\Omega_{t})$,

\item $L^{0}(\Omega)$: the space of all $\mathcal{B}(\Omega)$-measurable real functions,

\item $L^{0}(\Omega_{t})$: the space of all $\mathcal{B}(\Omega_{t}%
)$-measurable real functions,

\item $B_{b}(\Omega)$: all bounded elements in $L^{0}(\Omega)$, $B_{b}%
(\Omega_{t}):=B_{b}(\Omega)\cap L^{0}(\Omega_{t})$,

\item $C_{b}(\Omega)$: all continuous elements in $B_{b}(\Omega)$;
$C_{b}(\Omega_{t}):=B_{b}(\Omega)\cap L^{0}(\Omega_{t})$.
\end{itemize}

In \cite{Peng2006b, Peng2007}, a $G$-Brownian motion is constructed on a
sublinear expectation space $(\Omega,\mathbb{L}_{G}^{p}(\Omega),\mathbb{\hat
{E}})$ for $p\geq1$, with $\mathbb{L}_{G}^{p}(\Omega)$ such that
$\mathbb{L}_{G}^{p}(\Omega)$ is a Banach space under the natural norm
$\left \Vert X\right \Vert _{p}:=\mathbb{\hat{E}}[|X|^{p}]^{1/p}$. In this space
the corresponding canonical process $B_{t}(\omega)=\omega_{t}$, $t\in
\lbrack0,\infty)$, for $\omega \in \Omega$ is a $G$-Brownian motion.

Moreover, the notion of $G$-conditional expectation was also introduced
$\mathbb{\hat{E}}[\cdot|\Omega_{t}]:\mathbb{L}_{G}^{p}(\Omega)\mapsto
\mathbb{L}_{G}^{p}(\Omega_{t})$, for each $t\geq0$. It satisfies:
$\mathbb{\hat{E}}[${{$\mathbb{\hat{E}}[X|\Omega_{t}]\Omega_{s}]=$}%
}$\mathbb{\hat{E}}${{$[X|\Omega_{t\wedge s}]$ and}} {{\newline \textbf{(a)
Monotonicity:} \  \  \  \  \  \  \  \  \  \  \  \  \ If $X\geq Y$ then $\mathbb{\hat{E}%
}[X|\Omega_{t}]\geq{{\mathbb{\hat{E}}[Y|\Omega_{t}]}},$\newline \textbf{(b)
Constant preserving: \  \ }\  \ {$\mathbb{\hat{E}}[\eta|\Omega_{t}]$}$=$}}$\eta
${,\  \  \ }$\eta \in \mathbb{L}_{G}^{p}(\Omega_{t}),${{\newline \textbf{(c)}
\textbf{Sub-additivity: \  \  \  \ }}}\  \  \  \  \  \  \  \ {{$\mathbb{\hat{E}%
}[X|\Omega_{t}]$}}$-{{\mathbb{\hat{E}}[Y|\Omega_{t}]}}\leq{{\mathbb{\hat{E}%
}[X-Y|\Omega_{t}]}}.$\newline{{\textbf{(d) Positive homogeneity: }
\ {$\mathbb{\hat{E}}[\eta X|\Omega_{t}]$}$=\eta${$\mathbb{\hat{E}}%
[X|\Omega_{t}]$},$\  \ $for bounded and non-negative $\eta \in$}}$\mathbb{L}%
_{G}^{p}(\Omega_{t})${{.}}\newline Furthermore, it is proved in \cite{DHP}
(see also \cite{HP} for a simple proof) that $L^{0}(\Omega)\supset
\mathbb{L}_{G}^{p}(\Omega)\supset C_{b}(\Omega)$, and there exists a weakly
compact family $\mathcal{P}$ of probability measures defined on $(\Omega
,\mathcal{B}(\Omega))$ such that
\[
\mathbb{\hat{E}}[X]=\sup_{P\in \mathcal{P}}E_{P}[X],\  \text{for}\ X\in
C_{b}(\Omega).
\]
We introduce the natural Choquet capacity (see \cite{Choquet}):%
\[
\hat{c}(A):=\sup_{P\in \mathcal{P}}P(A),\  \ A\in \mathcal{B}(\Omega).
\]
The space $\mathbb{L}_{G}^{2}(\Omega)$ was also introduced independently in
\cite{Denis-M} in a quite different framework.

\begin{definition}
A set $A\subset \Omega$ is polar if $\hat{c}(A)=0$. A property holds
\textquotedblleft quasi-surely\textquotedblright \ (q.s.) if it holds outside a
polar set.
\end{definition}

$\mathbb{L}_{G}^{p}(\Omega)$ can be characterized as follows:
\[
\mathbb{L}_{G}^{p}(\Omega)=\{X\in \mathbb{L}^{0}(\Omega)|\sup_{P\in \mathcal{P}%
}E_{P}[|X|^{p}]<\infty \text{, and }X\text{ is }\hat{c}\text{-quasi surely
continuous}\}.\text{ }%
\]
We also denote, for $p>0$,

\begin{itemize}
\item $\mathcal{L}^{p}:=\{X\in L^{0}(\Omega):\mathbb{\hat{E}}[|X|^{p}%
]=\sup_{P\in \mathcal{P}}E_{P}[|X|^{p}]<\infty \}$;\

\item $\mathcal{N}^{p}:=\{X\in L^{0}(\Omega):\mathbb{\hat{E}}[|X|^{p}]=0\}$;

\item $\mathcal{N}:=\{X\in L^{0}(\Omega):X=0$, $\hat{c}$-quasi surely
(q.s.).$\}$
\end{itemize}

It is easy to see that $\mathcal{L}^{p}$ and $\mathcal{N}^{p}$ are linear
spaces and $\mathcal{N}^{p}=\mathcal{N}$, for each $p>0$. We denote by
$\mathbb{L}^{p}:=\mathcal{L}^{p}/\mathcal{N}$. As usual, we do not make the
distinction between classes and their representatives.

Now, we give the following two propositions which can be found in \cite{DHP}.

\begin{proposition}
\label{pr8} For each $\{X_{n} \}_{n=1}^{\infty}$ in $C_{b}(\Omega)$ such that
$X_{n}\downarrow0$ \ on $\Omega$, we have $\mathbb{\hat{E}} [X_{n}%
]\downarrow0$.
\end{proposition}

\begin{proposition}
\label{pr9} We have

\begin{enumerate}
\item For each $p\geq1$, $\mathbb{L}^{p}$ is a Banach space under the norm
$\left \Vert X\right \Vert _{p}:=\left(  \mathbb{\hat{E}}[|X|^{p}]\right)
^{\frac{1}{p}}$.

\item $\mathbb{L}_{\ast}^{p}$ is the completion of $B_{b}(\Omega)$ under the
Banach norm $\mathbb{\hat{E}}[|X|^{p}]^{1/p}$.

\item $\mathbb{L}_{G}^{p}$ is the completion of $C_{b}(\Omega)$.
\end{enumerate}
\end{proposition}

The following proposition is obvious.

\begin{proposition}
We have

\begin{enumerate}
\item $\mathbb{L}_{\ast}^{p}\subset \mathbb{L}^{p}\subset \mathbb{L}_{\ast}%
^{q}\subset \mathbb{L}^{q}$, $0<p\leq q\leq \infty$;

\item $\left \Vert X\right \Vert _{p}\uparrow \left \Vert X\right \Vert _{\infty}$,
for each $X\in \mathbb{L}^{\infty}$;

\item $p,q>1$, $\frac{1}{p}+\frac{1}{q}=1$. Then $X\in \mathbb{L}^{p}$ and
$Y\in \mathbb{L}^{q}$ implies
\[
XY\in \mathbb{L}^{1}\text{ and }\mathbb{E}[|XY|]\leq \left(  \mathbb{E}%
[|X|^{p}]\right)  ^{\frac{1}{p}}\left(  \mathbb{E}[|Y|^{q}]\right)  ^{\frac
{1}{q}}.
\]
\newline


\end{enumerate}
\end{proposition}

\begin{proposition}
For a given $p\in(0,+\infty]$, let $\{X_{n}\}_{n=1}^{\infty}$ be a sequence in
$\mathbb{L}^{p}$ which converges to $X$ in $\mathbb{L}^{p}$. Then there exists
a subsequence $(X_{n_{k}})$ which converges to $X$ quasi-surely in the sense
that it converges to $X$ outside a polar set.
\end{proposition}

We also have

\begin{proposition}
\label{Prop5}For each $p>0$,%
\[
\mathbb{L}_{\ast}^{p}=\{X\in \mathbb{L}^{p}:\lim_{n\rightarrow \infty}%
\mathbb{E}[|X|^{p}\mathbf{1}_{\{|X|>n\}}]=0\}.
\]

\end{proposition}

We introduce the following properties. They are important in this paper:

\begin{proposition}
\label{p1} For each $0\leq t<T$, $\xi \in \mathbb{L}^{2}(\Omega_{t})$, we have%
\[
\mathbb{\hat{E}}[\xi(B_{T}-B_{t})]=0.
\]

\end{proposition}

\begin{proof}
Let $P\in \mathcal{P}$ be given. If $\xi \in C_{b}(\Omega_{t})$, then we have
\[
0=-\mathbb{\hat{E}}[-\xi(B_{T}-B_{t})]\leq E_{P}[\xi(B_{T}-B_{t}%
)]\leq \mathbb{\hat{E}}[\xi(B_{T}-B_{t})]=0.
\]
In the case when $\xi \in \mathbb{L}^{2}(\Omega_{t})$, we have $E_{P}[|\xi
|^{2}]\leq \mathbb{\hat{E}}[|\xi|^{2}]<\infty$. Since it is known that
$C_{b}(\Omega_{t})$ is dense in $L_{P}^{2}(\Omega_{t})$, we then can choose a
sequence $\{ \xi_{n}\}_{n=1}^{\infty}$ in $C_{b}(\Omega_{t})$ such that
$E_{P}[|\xi-\xi_{n}|^{2}]\rightarrow0$. Thus
\[
E_{P}[\xi(B_{T}-B_{t})]=\lim_{n\rightarrow \infty}E_{P}[\xi_{n}(B_{T}%
-B_{t})]=0.
\]
The proof is complete.
\end{proof}

From now on and throughout this paper we restrict ourselves to the situation
of $1$-dimensional $G$-Brownian motion case. In this case $G(a)$ becomes a
given sublinear and monotonic real valued function defined on $\mathbb{R}$.
$G$ can be written as%
\[
G(a)=\frac{1}{2}(\overline{\sigma}^{2}a^{+}-\underline{\sigma}^{2}a^{-}).
\]

\begin{proposition}
\label{p2} For each $0\leq t\leq T$, $\xi \in B_{b}({\Omega_{t}})$, we have
\begin{equation}
\hat{\mathbb{E}}[\xi^{2}(B_{T}-B_{t})^{2}-\overline{\sigma}^{2}\xi
^{2}(T-t)]\leq0. \label{PropP1}%
\end{equation}

\end{proposition}

\begin{proof}
If $\xi \in C_{b}(\Omega_{t})$, then by \cite{Peng2006a}, we have the following
It\^{o} formula:
\[
\xi^{2}[(B_{T}-B_{t})^{2}-(\langle B_{T}\rangle-\langle B_{t}\rangle
)]=2\int_{t}^{T}\xi^{2}B_{s}dB_{s}.\text{ }%
\]
It follows that $\mathbb{\hat{E}[}\xi^{2}(B_{T}-B_{t})^{2}-\xi^{2}(\langle
B_{T}\rangle-\langle B_{t}\rangle)]=0$. On the other hand, we have $\langle
B\rangle_{T}-\langle B\rangle_{t}\leq \bar{\sigma}^{2}(T-t)$, quasi surely.
Thus (\ref{PropP1}) holds for $\xi \in C_{b}(\Omega_{t})$. It follows that, for
each fixed $P\in \mathcal{P}$, we have
\begin{equation}
E_{P}\mathbb{[}\xi^{2}(B_{T}-B_{t})^{2}-\xi^{2}(\langle B\rangle_{T}-\langle
B\rangle_{t})]\leq0. \label{PropP1-1}%
\end{equation}
In the case when $\xi \in B_{b}(\Omega_{t})$, we can find a sequence $\{
\xi_{n}\}_{n=1}^{\infty}$ in $C_{b}(\Omega_{t})$, such that $\xi
_{n}\rightarrow \xi$ in $L^{p}(\Omega,\mathcal{F}_{t},P)$, for some $p>2$. Thus
we have
\[
E_{P}\mathbb{[}\xi_{n}^{2}(B_{T}-B_{t})^{2}-\xi_{n}^{2}(\langle B_{T}%
\rangle-\langle B_{t}\rangle)]\leq0,
\]
and then, by letting $n\rightarrow \infty$, we obtain (\ref{PropP1-1}) for
$\xi \in B_{b}(\Omega_{t})$. Thus (\ref{PropP1}) follows immediately for
$\xi \in B_{b}(\Omega_{t})$.
\end{proof}

The space $L_{G}^{p}(\Omega)$ is a Banach space under the norm $\left \Vert
\cdot \right \Vert _{p}:=\mathbb{E}_{G}[|\cdot|^{p}]^{1/p}$. We have also
introduced a space of `adapted processes' $M_{G}^{p}(0,T)$ which is also a
Banach space under the following norm:%
\[
\left \Vert \eta \right \Vert _{M_{G}^{p}(0,T)}=\left(  \int_{0}^{T}%
\mathbb{E}_{G}[|\eta_{t}|^{p}]dt\right)  ^{1/p},\  \  \eta \in M_{G}^{p}(0,T).
\]
Exactly following It\^{o}'s original idea, for each $\eta \in M_{G}^{2}(0,T)$,
we have introduced It\^{o}'s integral%
\[
\int_{0}^{T}\eta_{t}dB_{t}.
\]
We have%
\[
\mathbb{E}_{G}[\int_{0}^{T}\eta_{t}dB_{t}]=0,\  \  \mathbb{E}_{G}\left[  \left(
\int_{0}^{T}\eta_{t}dB_{t}\right)  ^{2}\right]  =\mathbb{E}_{G}\left[
\int_{0}^{T}|\eta_{t}|^{2}d\left \langle B\right \rangle _{t}\right]  .
\]
Moreover for each $\zeta \in M_{G}^{1}(0,T)$ we have the following estimate.
For each $\eta \in \mathbb{M}_{G}^{1}(0,T)$ we have%
\begin{equation}
\overline{\sigma}^{2}\mathbb{E}_{G}[\int_{0}^{T}|\zeta_{s}|ds]\geq
\mathbb{E}_{G}[\int_{0}^{T}|\zeta_{s}|d\left \langle B\right \rangle _{s}%
]\geq \underline{\sigma}^{2}\mathbb{E}_{G}[\int_{0}^{T}|\zeta_{s}%
|ds].\label{0.5}%
\end{equation}
The following relations play an essentially important role in this paper (see
\cite{Peng2010}): for each $\zeta \in M_{G}^{1}(0,T)$,%
\begin{equation}
\int_{0}^{T}\zeta_{t}d\left \langle B\right \rangle _{t}-\int_{0}^{T}%
2G(\zeta_{t})dt\leq0,\  \mathbb{E}_{G}[\int_{0}^{T}\zeta_{s}d\left \langle
B\right \rangle _{s}-\int_{0}^{T}2G(\zeta_{s})ds]=0.\label{0.10}%
\end{equation}
In this paper we introduce a norm $\left \Vert \cdot \right \Vert _{\mathbb{M}%
_{G}^{p}(0,T)}$ which is weaker than $\left \Vert \cdot \right \Vert _{M_{G}%
^{p}(0,T)}$:%
\[
\left \Vert \eta \right \Vert _{\mathbb{M}_{G}^{p}(0,T)}=\left(  \mathbb{E}%
_{G}[\int_{0}^{T}|\eta_{t}|^{p}dt]\right)  ^{1/p}.
\]
The completion of $M_{G}^{p}(0,T)$ under this norm is denoted by
$\mathbb{M}_{G}^{p}(0,T)$. It is easy to check that the definition of
It\^{o}'s integral can be extended to the case $\eta \in \mathbb{M}_{G}%
^{2}(0,T)$ and the integral $\int_{0}^{T}\zeta_{t}d\left \langle B\right \rangle
_{t}$ can be defined also for $\zeta \in \mathbb{M}_{G}^{1}(0,T)$. The above
relations still hold true.

We list the result of Song of $G$-martingale decomposition which generalizes
that of \cite{STZ}. Let $H_{G}^{0}(0,T)$ be the family of simple processes of
form $z_{t}=\sum_{i=0}^{N-1}z_{i}\mathbf{1}_{[t_{i},t_{i+1})}(t)$,
$t\in \lbrack0,T]$, $z_{i}\in L_{ip}(\Omega_{t_{i}})$. For each $z$ in this
space, we define the following norm
\[
\left \Vert z\right \Vert _{H_{G}^{p}(0,T)}=\left[  \left(  \int_{0}^{T}%
|z_{s}|^{2}ds\right)  ^{p/2}\right]  ^{1/p}%
\]
and denote by $H_{G}^{p}(0,T)$ the completion of $H_{G}^{0}(0,T)$ under this norm.

\begin{theorem}
\label{Songthm} (\cite{Song}, Theorem 4.5) For any given $p>1$ and $X\in
L_{G}^{p}(\Omega_{T})$ the $G$-martingale $M_{t}=\mathbb{E}_{G}[X|\Omega_{t}%
]$, $t\in \lbrack0,T]$, has the following decomposition:%
\[
M_{t}=\mathbb{E}_{G}[X]+\int_{0}^{t}z_{s}dB_{s}-A_{t},\ t\in \lbrack0,T],
\]
where $z\in H_{G}^{1}(0,T)$ and $A$ is a continuous increasing process with
$A_{0}=0$ such that $(-A_{t})_{0\leq t\leq T}$ is a $G$-martingale.
Furthermore the above decomposition is unique and $z\in H_{G}^{\alpha}(0,T)$,
$K_{T}\in L_{G}^{\alpha}(\Omega_{T})$ for any $1\leq \alpha<p$.
\end{theorem}

\section{A priori estimates and representation theorem for $L_{ip}(\Omega
_{T})$}

Let the function $G$ be given as the above. Then for a fixed $\varepsilon
\in(0,(\overline{\sigma}^{2}-\underline{\sigma}^{2})/2]$, we set
$G_{\varepsilon}(a)=G(a)-\frac{\varepsilon}{2}|a|$.

\begin{theorem}
We assume that $\xi \in L_{G}^{2}(\Omega_{T})$ has the following
representation: there exists a pair of processes $(z,\eta)\in H_{G}%
^{p}(0,T)\times M_{G_{\varepsilon}}^{1}(0,T)$, with $p\in \lbrack1,2)$, such
that
\begin{equation}
\xi=\mathbb{E}_{G}[\xi]+\int_{0}^{T}z_{s}^{\xi}dB_{s}+\int_{0}^{T}\eta
_{s}^{\xi}d\left \langle B\right \rangle _{s}-\int_{0}^{T}2G(\eta_{s}^{\xi})ds.
\label{2.8}%
\end{equation}
Then we have
\begin{equation}
\varepsilon \mathbb{E}_{G_{\varepsilon}}[\int_{0}^{T}|\eta_{s}^{\xi}%
|ds]\leq \mathbb{E}_{G}[\xi]+\mathbb{E}_{G_{\varepsilon}}[-\xi]. \label{2.5}%
\end{equation}

\end{theorem}

\begin{proof}
From
\begin{align*}
&  0=\mathbb{E}_{G_{\varepsilon}}[\int_{0}^{T}z_{s}^{\xi}dB_{s}+\int_{0}%
^{T}\eta_{s}^{\xi}d\left \langle B\right \rangle _{s}-\int_{0}^{T}%
2G_{\varepsilon}(\eta_{s}^{\xi})ds]\\
&  =\mathbb{E}_{G_{\varepsilon}}[\xi-\mathbb{E}_{G}[\xi]+\int_{0}^{T}%
2(G(\eta_{s}^{\xi})-G_{\varepsilon}(\eta_{s}^{\xi}))ds]\\
&  =\mathbb{E}_{G_{\varepsilon}}[\xi-\mathbb{E}_{G}[\xi]+\int_{0}%
^{T}\varepsilon|\eta_{s}^{\xi}|ds]\\
&  \geq \mathbb{E}_{G_{\varepsilon}}[\int_{0}^{T}\varepsilon|\eta_{s}^{\xi
}|ds]-\mathbb{E}_{G}[\xi]-\mathbb{E}_{G_{\varepsilon}}[-\xi]
\end{align*}
we immediately have (\ref{2.5}).
\end{proof}

For the rest of this paper we assume that $\overline{\sigma}\geq
\underline{\sigma}>0$.

We set
\[
L_{ip}(\Omega_{T})=\{ \xi=\varphi(B_{t_{1}},B_{t_{2}},\cdots,B_{t_{n}%
}),\  \varphi \in C_{Lip}(\mathbb{R}^{n}),\ t_{i}\in \lbrack0,\infty)\text{,
}i=1,2,\cdots,n\geq1\}
\]

\begin{theorem}
For each $\xi \in L_{ip}(\Omega_{T})$, we have a unique representation
(\ref{2.8}).
\end{theorem}

\begin{proof}
It suffices to prove the existence of $(z,\eta)$ for such type of $\xi$.

We first consider the case where $\xi=\varphi(B_{t_{2}}-B_{t_{1}})$ with
$t_{1}<t_{2}\leq T$ and $\varphi$ is a Lipschitz and bounded function on
$\mathbb{R}$. Let $V$ be the unique viscosity solution of%
\begin{equation}
\partial_{t}V+G(D^{2}V)=0,\ (t,x)\in \lbrack t_{1},t_{2})\times \mathbb{R}%
^{d}\text{,}\  \label{e14}%
\end{equation}
with terminal condition
\begin{equation}
V(t_{2},x)=\varphi(x).\label{equ-h}%
\end{equation}
Since (\ref{e14}) is a uniform parabolic PDE and $G$ is a convex and Lipschitz
function thus, by the regularity of $V$ (see Krylov \cite{Krylov}, Example
6.1.8 and Theorem 6.4.3, see also Appendix of \cite{Peng2010}), we have
\[
\left \Vert V\right \Vert _{C^{1+\alpha/2,2+\alpha}([t_{1},t_{2}-\delta
]\times \mathbb{R}^{d})}<\infty,\  \text{for some }\alpha \in(0,1).
\]
Then, for each $\delta \in(0,t_{2}-t_{1})$, we set
\[
\eta_{s}^{\delta}=\mathbf{1}_{[t_{1},t_{2}-\delta]}(s)\frac{1}{2}%
V_{xx}(s,B_{s}),\  \ z_{s}^{\delta}=\mathbf{1}_{[t_{1},t_{2}-\delta]}%
(s)V_{x}(s,B_{s}),\  \ s\in \lbrack0,T],
\]
and%
\[
\eta_{s}=\mathbf{1}_{[t_{1},t_{2})}(s)\frac{1}{2}V_{xx}(s,B_{s}),\  \ z_{s}%
=\mathbf{1}_{[t_{1},t_{2})}(s)V_{x}(s,B_{s}),\  \ s\in \lbrack0,T].
\]
It is clear that $(z^{\delta},\eta^{\delta})\in M_{\bar{G}}^{2}(0,T)$ for each
fixed $\delta \in(0,t_{2}-t_{1})$. By It\^{o}'s formula we have \
\begin{align*}
&  V(t_{2}-\delta,B_{t_{2}-\delta})-V(t_{1},0)\\
&  =\int_{0}^{T}z_{s}^{\delta}dB_{s}+\int_{0}^{T}\eta_{s}^{\delta
}d\left \langle B\right \rangle _{s}-\int_{0}^{T}2G(\eta_{s}^{\delta})ds.
\end{align*}
On the other hand, for each $t,t^{\prime}\in \lbrack t_{1},t_{2}]$ and
$x,x^{\prime}\in \mathbb{R}^{d}$, we have
\[
|V(t,x)-V(t^{\prime},x^{\prime})|\leq C(\sqrt{|t-t^{\prime}|}+|x-x^{\prime}|).
\]
Thus we have
\begin{align*}
&  |\mathbb{\hat{E}}[|V(t_{2}-\delta^{\prime},B_{t_{2}-\delta^{\prime}%
})-V(t_{2}-\delta,B_{t_{2}-\delta})|^{2}]\\
&  \leq C(\sqrt{\delta-\delta^{\prime}}+|\delta-\delta^{\prime}|)+C\mathbb{E}%
_{G}[|B_{t_{2}-\delta}-B_{t_{2}-\delta^{\prime}}|^{2}]\\
&  \leq \bar{C}(\sqrt{\delta-\delta^{\prime}}+|\delta-\delta^{\prime}|).
\end{align*}
Now for each $t_{2}-t_{1}>\delta>\delta^{\prime}>0$,
\begin{align*}
&  V(t_{2}-\delta,B_{t_{2}-\delta})-V(t_{2}-\delta^{\prime},B_{t_{2}%
-\delta^{\prime}})\\
&  =\int_{t_{1}}^{t_{2}}(z_{s}^{\delta}-z_{s}^{\delta^{\prime}})dB_{s}%
+\int_{t_{1}}^{t_{2}}(\eta_{s}^{\delta}-\eta_{s}^{\delta^{\prime}%
})d\left \langle B\right \rangle _{s}-2\int_{t_{1}}^{t_{2}}(G(\eta_{s}^{\delta
})-G(\eta_{s}^{\delta^{\prime}}))ds\\
&  =\int_{t_{1}}^{t_{2}}(z_{s}^{\delta}-z_{s}^{\delta^{\prime}})dB_{s}%
+\int_{t_{1}}^{t_{2}}(\eta_{s}^{\delta}-\eta_{s}^{\delta^{\prime}%
})d\left \langle B\right \rangle _{s}-2\int_{t_{1}}^{t_{2}}G(\eta_{s}^{\delta
}-\eta_{s}^{\delta^{\prime}})ds.
\end{align*}
We then can apply (\ref{2.5}) to prove that%
\begin{align*}
\varepsilon \mathbb{E}_{G_{\varepsilon}}[\int_{t_{1}}^{t_{2}}|\eta_{s}^{\delta
}-\eta_{s}^{\delta^{\prime}}|ds] &  \leq \mathbb{E}_{G}[V(t_{2}-\delta
,B_{t_{2}-\delta})-V(t_{2}-\delta^{\prime},B_{t_{2}-\delta^{\prime}})]\\
&  +\mathbb{E}_{G_{\varepsilon}}[-V(t_{2}-\delta,B_{t_{2}-\delta}%
)+V(t_{2}-\delta^{\prime},B_{t_{2}-\delta^{\prime}})]\\
&  \leq \bar{C}(\sqrt{\delta-\delta^{\prime}}+|\delta-\delta^{\prime}|).
\end{align*}
Thus, when $\delta_{i}\downarrow0$, $\{ \eta^{\delta_{i}}\}_{i=1}^{\infty}$
forms a Cauchy sequence in $M_{G_{\varepsilon}}^{1}(0,T)$, and the
representation of $\xi=\varphi(B_{t_{2}}-B_{t_{1}})$ is uniquely given
by{\footnotesize { }}%
\[
\xi=V(t_{1},0)+\int_{0}^{T}z_{s}dB_{s}+\int_{0}^{T}\eta_{s}d\left \langle
B\right \rangle _{s}-\int_{0}^{T}2G(\eta_{s})ds.
\]
Moreover we have $V(t_{1},0)=\mathbb{E}_{G}[X]=\mathbb{E}_{G}[\varphi(B_{T})]$.

We now consider the case where $\xi=\varphi(B_{t_{1}},B_{t_{2}}-B_{t_{1}})$,
where $\varphi$ is a Lipschitz and bounded function on $\mathbb{R}^{2}$. For
this random variable we first solve the following PDE for a fixed parameter
$y\in \mathbb{R}$:%
\[
\partial_{t}V^{y}+G(V_{xx}^{y})=0,\ s\in \lbrack t_{1},t_{2}),\  \ V^{y}%
(t_{2},x)=\varphi(y,x).\
\]
Then, setting
\[
\eta_{s}^{\delta}=\mathbf{1}_{[t_{1},t_{2}-\delta]}(s)\frac{1}{2}%
V_{xx}^{B_{t_{1}}}(s,B_{s}-B_{t_{1}}),\  \ z_{s}^{\delta}=\mathbf{1}%
_{[t_{1},t_{2}-\delta]}(s)V_{x}^{B_{t_{1}}}(s,B_{s}-B_{t_{1}}),\  \ s\in \lbrack
t_{1},T).
\]
and%
\[
\eta_{s}=\mathbf{1}_{[t_{1},t_{2})}(s)\frac{1}{2}V_{xx}^{B_{t_{1}}}%
(s,B_{s}-B_{t_{1}}),\  \ z_{s}=\mathbf{1}_{[t_{1},t_{2})}(s)V_{x}^{B_{t_{1}}%
}(s,B_{s}-B_{t_{1}}),\  \ s\in \lbrack t_{1},T].
\]
Exactly as in the first case, we can prove that
\begin{equation}
\xi=\mathbb{E}_{G}[\xi|\Omega_{t_{1}}]+\int_{t_{1}}^{T}z_{s}dB_{s}+\int
_{t_{1}}^{T}\eta_{s}d\left \langle B\right \rangle _{s}-\int_{t_{1}}^{T}%
2G(\eta_{s})ds \label{eq2.0}%
\end{equation}
and, moreover $\mathbb{E}_{G}[X|\Omega_{t_{1}}]=V^{B_{t_{1}}}(t_{1},0)$. It is
easy to check that $V^{y}(t,0)$ is a bounded and Lipschitz function of
$y\in \mathbb{R}$. We then can further solve, backwardly,
\[
\partial_{t}\bar{V}+G(\bar{V}_{xx})=0,\  \ t\in \lbrack0,t_{1}],\ V(t,x)=V^{x}%
(t,0).
\]
Setting
\[
\eta_{s}^{\delta}=\mathbf{1}_{[0,t_{1}-\delta]}(s)\frac{1}{2}\bar{V}%
_{xx}(s,B_{s}),\  \ z_{s}^{\delta}=\mathbf{1}_{[0,t_{1}-\delta]}(s)\bar{V}%
_{x}(s,B_{s}),\  \ s\in \lbrack0,t),
\]
and%
\[
\eta_{s}=\mathbf{1}_{[0,t_{1})}(s)\frac{1}{2}\bar{V}_{xx}(s,B_{s}%
),\  \ z_{s}=\mathbf{1}_{[0,t_{1})}(s)\bar{V}_{x}(s,B_{s}),\  \ s\in
\lbrack0,t_{1}),
\]
and then using again the same approach, we arrive at
\begin{equation}
\mathbb{E}_{G}[\xi|\Omega_{t_{1}}]=\mathbb{E}_{G}[\xi]+\int_{0}^{t_{1}}%
z_{s}dB_{s}+\int_{0}^{t_{1}}\eta_{s}d\left \langle B\right \rangle _{s}-\int
_{0}^{t_{1}}2G(\eta_{s})ds. \label{eq2.1}%
\end{equation}
This with (\ref{eq2.0}) yields that
\[
\xi=\mathbb{E}_{G}[\xi]+\int_{0}^{T}z_{s}dB_{s}+\int_{0}^{T}\eta
_{s}d\left \langle B\right \rangle _{s}-\int_{0}^{T}2G(\eta_{s})ds.
\]

We can use exactly the same approach to find, for $X=\varphi(B_{t_{1}%
},B_{t_{2}}-B_{t_{1}},\cdots,B_{t_{n}}-B_{t_{n-1}})$, for any given bounded
and Lipschitz function $\varphi$ defined on $\mathbb{R}^{n}$ and for any
$0\leq t_{1}<t_{2}<\cdots<t_{n}\leq T$. The proof is complete.
\end{proof}

\section{Uniqueness of representation theorem}

\begin{lemma}
We assume $\xi \in L_{\overline{G}}^{2}(\Omega_{T})$, with $\overline{G}%
=G\vee \overline{G}_{\varepsilon}$ where%
\begin{align*}
\overline{G}_{\varepsilon}(a)  &  =\frac{1}{2}[(\underline{\sigma}%
^{2}-\varepsilon)a^{+}-\underline{\sigma}_{\varepsilon}^{2}a^{-}%
],\  \ 0<\underline{\sigma}_{\varepsilon}^{2}\leq(\underline{\sigma}%
^{2}-\varepsilon)\\
G^{\varepsilon}(a)  &  =\frac{1}{2}[\overline{\sigma}_{\varepsilon}^{2}%
a^{+}-(\overline{\sigma}^{2}+\varepsilon)a^{-}],\  \  \ (\overline{\sigma}%
^{2}+\varepsilon)\leq \overline{\sigma}_{\varepsilon}^{2}.
\end{align*}
Then there exists at most one $(z,\eta)\in M_{\overline{G}_{\varepsilon}}%
^{2}(0,T)\times M_{\overline{G}_{\varepsilon}}^{1}(0,T)$ satisfying
\[
\xi=\mathbb{E}_{G}[\xi]+\int_{0}^{T}z_{s}dB_{s}+\int_{0}^{T}\eta
_{s}d\left \langle B\right \rangle _{s}-\int_{0}^{T}2G(\eta_{s})ds.
\]

\end{lemma}

\begin{proof}
Let $(z^{i},\eta^{i})\in M_{\overline{G}_{\varepsilon}}^{2}(0,T)\times
M_{\overline{G}_{\varepsilon}}^{1}(0,T)$, $i=1,2$ satisfy the above
representation. Then
\[
\int_{0}^{t}\hat{z}_{s}dB_{s}+\int_{0}^{t}\hat{\eta}_{s}d\left \langle
B\right \rangle _{s}-\int_{0}^{t}2[G(\eta_{s}^{1})-G(\eta_{s}^{2})]ds\equiv0,
\]
where we denote $\hat{z}=z^{1}-z^{2}$, $\hat{\eta}=\eta^{1}-\eta^{2}$. For
each bounded process $\mu \in M_{G}^{2}(0,T)$, we have%
\[
\int_{0}^{T}\mu_{s}\hat{z}_{s}dB_{s}+\int_{0}^{T}\mu_{s}\hat{\eta}%
_{s}d\left \langle B\right \rangle _{s}-\int_{0}^{t}2\mu_{s}[G(\eta_{s}%
^{1})-G(\eta_{s}^{2})]ds\equiv0.
\]
Fix $\delta>0$, we put
\[
\mu_{s}=\frac{\hat{\eta}_{s}}{\delta+|\hat{\eta}_{s}|}.
\]
Since $(a-b)(G(a)-G(b)=|a-b|\cdot|G(a)-G(b)|$,%
\begin{align*}
I_{\varepsilon} &  =\int_{0}^{T}\mu_{s}\hat{z}_{s}dB_{s}+\int_{0}^{T}%
\frac{|\hat{\eta}_{s}|^{2}}{\delta+|\hat{\eta}_{s}|}d\left \langle
B\right \rangle _{s}-\int_{0}^{T}2\overline{G}_{\varepsilon}(\frac{|\hat{\eta
}_{s}|^{2}}{\delta+|\eta_{s}|})ds\\
&  =\int_{0}^{T}[\frac{2|G(\eta_{s}^{1})-G(\eta_{s}^{2})|\cdot|\hat{\eta}%
_{s}|}{\delta+|\hat{\eta}_{s}|}ds-2\overline{G}_{\varepsilon}(\frac{|\hat
{\eta}_{s}|^{2}}{\delta+|\eta_{s}|})]ds\\
&  \geq \int_{0}^{T}[\frac{\underline{\sigma}^{2}|\hat{\eta}_{s}|^{2}}%
{\delta+|\hat{\eta}_{s}|}-(\underline{\sigma}^{2}-\varepsilon)\frac{|\hat
{\eta}_{s}|^{2}}{\delta+|\eta_{s}|}]ds\\
&  =\varepsilon \int_{0}^{T}\frac{|\hat{\eta}_{s}|^{2}}{\delta+|\hat{\eta}%
_{s}|}ds.
\end{align*}
Since $\mathbb{E}_{\overline{G}_{\varepsilon}}[I_{\varepsilon}]=0$, we then
have%
\[
\mathbb{E}_{\overline{G}_{\varepsilon}}\mathbb{[}\int_{0}^{T}\frac{|\hat{\eta
}_{s}|^{2}}{\delta+|\hat{\eta}_{s}|}ds]=0.
\]
Sending $\delta$ to $0$ we deduce that $\mathbb{E}_{\overline{G}_{\varepsilon
}}[\int_{0}^{T}|\hat{\eta}_{s}|ds]=0$. We thus have $\hat{\eta}_{s}\equiv0$ in
$M_{\overline{G}_{\varepsilon}}^{1}(0,T)$. Thus
\[
\mathbb{E}_{\overline{G}_{\varepsilon}}[\int_{0}^{T}|z_{s}^{1}-z_{s}^{2}%
|^{2}ds]\leq \frac{1}{\underline{\sigma}_{\varepsilon}^{2}}\mathbb{E}%
_{\overline{G}_{\varepsilon}}[\left(  \int_{0}^{T}(z_{s}^{1}-z_{s}^{2}%
)dB_{s}\right)  ^{2}]=0.
\]
From which we have $z^{1}=z^{2}$.
\end{proof}

\begin{remark}
We can also take
\[
\mu_{s}=\frac{-\hat{\eta}_{s}}{\delta+|\hat{\eta}_{s}|}%
\]
In this case we have%
\begin{align*}
I  &  =\int_{0}^{T}\mu_{s}\hat{z}_{s}dB_{s}+\int_{0}^{T}\frac{-|\hat{\eta}%
_{s}|^{2}}{\delta+|\hat{\eta}_{s}|}d\left \langle B\right \rangle _{s}-\int
_{0}^{T}2G^{\varepsilon}(\frac{-|\hat{\eta}_{s}|^{2}}{\delta+|\eta_{s}|})ds\\
&  =\int_{0}^{T}[\frac{-2|G(\eta_{s}^{1})-G(\eta_{s}^{2})|\cdot|\hat{\eta}%
_{s}|}{\delta+|\hat{\eta}_{s}|}ds-2G^{\varepsilon}(\frac{-|\hat{\eta}_{s}%
|^{2}}{\delta+|\eta_{s}|})]ds\\
&  \geq \int_{0}^{T}[\frac{-\overline{\sigma}^{2}|\hat{\eta}_{s}|^{2}}%
{\delta+|\hat{\eta}_{s}|}+(\overline{\sigma}^{2}+\varepsilon)\frac{|\hat{\eta
}_{s}|^{2}}{\delta+|\eta_{s}|}]ds\\
&  =\varepsilon \int_{0}^{T}\frac{|\hat{\eta}_{s}|^{2}}{\delta+|\hat{\eta}%
_{s}|}ds.
\end{align*}
Since $\mathbb{E}_{G^{\varepsilon}}[I]=0$, we then have%
\[
\varepsilon \mathbb{E}_{G^{\varepsilon}}\mathbb{[}\int_{0}^{T}\frac{|\hat{\eta
}_{s}|^{2}}{\delta+|\hat{\eta}_{s}|}ds]=0.
\]
Sending $\delta$ to $0$ we deduce that $\varepsilon \mathbb{E}_{G^{\varepsilon
}}[\int_{0}^{T}|\hat{\eta}_{s}|ds]=0$.
\end{remark}

\section{Existence of the representation}

We will use Theorem \ref{Songthm}. For each $\xi \in L_{G}^{\alpha}(\Omega
_{T})$, there exists a unique decomposition
\[
\xi=\mathbb{E}_{G}[\xi]+\int_{0}^{T}z_{s}^{\xi}dB_{s}-A_{T}^{\xi}%
\]
with $A_{0}^{\xi}=0$, $A_{t}^{\xi}-A_{s}^{\xi}\geq0$ and $\mathbb{E}%
_{G}[-A_{t}^{\xi}|\Omega_{s}]=-A_{s}^{\xi}$, for $s\leq t$.

\begin{definition}
We define the following distance in $L_{\overline{G}}^{2}({\Omega}_{T})$:
given $\alpha \in(1,2)$,%
\begin{align*}
\rho(\xi_{1},\xi_{2})  &  =\mathbb{E}_{\overline{G}}[|\xi_{1}-\xi_{2}%
|^{2}]^{1/2}+\mathbb{E}_{\overline{G}}[\sup_{\pi_{N}[0,T]}|\sum_{i=0}%
^{N}|A_{t_{i+1}^{N}}^{\xi_{1}}-A_{t_{i}^{N}}^{\xi_{1}}-(A_{t_{i+1}^{N}}%
^{\xi_{2}}-A_{t_{i}^{N}}^{\xi_{2}})|^{\alpha}]^{1/\alpha}\\
&  <\infty,\  \xi_{1},\xi_{2}\in L_{\overline{G}}^{2}({\Omega}_{T}).\
\end{align*}
where $\pi_{N}[0,T]$ is the collection of all (deterministic) finite
partitions of $[0,T]$.
\end{definition}

Since $(L_{ip}(\Omega_{T}),\rho)$ forms a metric space, we denote the
completion of this space under $\rho$ by $L_{\overline{G}}^{2\ast}(\Omega
_{T})$. For each $\xi \in L_{\overline{G}}^{2\ast}(\Omega_{T})$, there exists a
sequence $\{ \xi_{n}\}_{n=1}^{\infty}$ in $L_{ip}(\Omega_{T})$ such
$\lim_{n\rightarrow \infty}\rho(\xi,\xi_{n})=0$. We have
\[
\xi=\mathbb{E}_{G}[\xi]+\int_{0}^{T}z_{s}^{\xi}dB_{s}-A_{T}^{\xi}.
\]

The following Lemma is easy:

\begin{lemma}
For each $\xi \in L_{G}^{2}(\Omega_{T})$ and $\mu \in M_{G}^{2}(0,T)$ such that
$|\mu_{t}|\leq c$, we have%
\[
\sup_{\substack{\mu \in M_{G}^{2}(0,T)\\|\mu|\leq c}}\mathbb{E}_{G}[|\int
_{0}^{T}\mu_{s}dA_{t}^{\xi}-\int_{0}^{T}\mu_{s}dA_{t}^{\xi^{\prime}}|]\leq
c\rho(\xi,\xi^{\prime}).
\]

\end{lemma}

\begin{theorem}
For each $\xi \in L_{\overline{G}}^{2,\ast}(\Omega)$ and for each $\alpha
\in(1,2)$, there exists a $z^{\xi}\in H_{\overline{G}_{\varepsilon}}^{\alpha
}(0,T)$ and $\eta^{\xi}\in M_{\overline{G}_{\varepsilon}}^{1}(0,T)$ such that
\begin{equation}
\xi=\mathbb{E}_{G}[\xi]+\int_{0}^{T}z_{s}^{\xi}dB_{s}+\int_{0}^{T}\eta
_{s}^{\xi}d\left \langle B\right \rangle _{s}-\int_{0}^{T}2G(\eta_{s}^{\xi})ds.
\label{Xi}%
\end{equation}

\end{theorem}

\begin{proof}
Let $\{ \xi_{n}\}_{n=1}^{\infty}$ be a Cauchy sequence in the metric space
$(L_{G}^{2,\ast}(\Omega_{T}),\rho)$ such that $\rho(\xi_{n},\xi)\rightarrow0$
as $n\rightarrow \infty$. We have the following unique representation:%
\begin{equation}
\xi_{n}=\mathbb{E}_{G}[\xi_{n}]+\int_{0}^{T}z_{s}^{\xi_{n}}dB_{s}+\int_{0}%
^{T}\eta_{s}^{\xi_{n}}d\left \langle B\right \rangle _{s}-\int_{0}^{T}%
2G(\eta_{s}^{\xi_{n}})ds,\label{Xi-m}%
\end{equation}
with $(z^{\xi_{n}},\eta^{\xi_{n}})\in H_{\overline{G}_{\varepsilon}}^{\alpha
}(0,T)\times M_{\overline{G}_{\varepsilon}}^{1}(0,T)$, $i=1,2$ be have the
above representation. Then
\begin{align*}
\mathbb{E}_{G}[\xi_{m}|\Omega_{t}]-\mathbb{E}_{G}[\xi_{n}|\Omega_{t}] &
=\mathbb{E}_{G}[\xi_{m}]-\mathbb{E}_{G}[\xi_{n}]\\
&  +\int_{0}^{t}\hat{z}_{s}^{m,n}dB_{s}+\int_{0}^{t}\hat{\eta}_{s}%
^{m,n}d\left \langle B\right \rangle _{s}-\int_{0}^{t}2[G(\eta_{s}^{m}%
)-G(\eta_{s}^{n})]ds.
\end{align*}
where we denote $\hat{z}^{m,n}=z^{m}-z^{n}$, $\hat{\eta}=\eta^{m}-\eta^{n}$.
Fix $\delta>0$, we set
\[
\mu_{s}^{m,n}=\frac{\hat{\eta}_{s}^{m,n}}{\delta+|\hat{\eta}_{s}^{m,n}|}%
\]
Since $(a-b)(G(a)-G(b)=|a-b|\cdot|G(a)-G(b)|$,
\begin{align*}
I_{\varepsilon}^{m,n} &  =\int_{0}^{T}\frac{|\hat{\eta}_{s}^{m,n}|^{2}}%
{\delta+|\hat{\eta}_{s}^{m,n}|}d\left \langle B\right \rangle _{s}-\int_{0}%
^{T}2\overline{G}_{\varepsilon}(\frac{|\hat{\eta}_{s}^{m,n}|^{2}}{\delta
+|\eta_{s}^{m,n}|})ds\\
&  =\int_{0}^{T}[\frac{2|G(\eta_{s}^{m})-G(\eta_{s}^{n})|\cdot|\hat{\eta}%
_{s}^{m,n}|}{\delta+|\hat{\eta}_{s}^{m,n}|}ds-2\overline{G}_{\varepsilon
}(\frac{|\hat{\eta}_{s}^{m,n}|^{2}}{\delta+|\eta_{s}^{m,n}|})]ds\\
&  -\int_{0}^{T}\mu_{s}^{m,n}d(A_{t}^{\xi_{m}}-dA_{t}^{\xi_{n}})\\
&  \geq \int_{0}^{T}[\frac{\underline{\sigma}^{2}|\hat{\eta}_{s}^{m,n}|^{2}%
}{\delta+|\hat{\eta}_{s}^{m,n}|}-(\underline{\sigma}^{2}-\varepsilon
)\frac{|\hat{\eta}_{s}^{m,n}|^{2}}{\delta+|\eta_{s}^{m,n}|}]ds-\int_{0}^{T}%
\mu_{s}^{m,n}d(A_{t}^{\xi_{m}}-dA_{t}^{\xi_{n}})\\
&  =\int_{0}^{T}\frac{\varepsilon|\hat{\eta}_{s}^{m,n}|^{2}}{\delta+|\hat
{\eta}_{s}^{m,n}|}ds-\int_{0}^{T}\mu_{s}^{m,n}d(A_{t}^{\xi_{m}}-dA_{t}%
^{\xi_{n}})
\end{align*}
Since $\mathbb{E}_{\overline{G}_{\varepsilon}}[I_{\varepsilon}^{m,n}]=0$, we
then have%
\begin{align*}
0 &  \geq \mathbb{E}_{\overline{G}_{\varepsilon}}\mathbb{[}\int_{0}^{T}%
\frac{\varepsilon|\hat{\eta}_{s}^{m,n}|^{2}}{\delta+|\hat{\eta}_{s}^{m,n}%
|}ds-\int_{0}^{T}\mu_{s}^{m,n}d(A_{t}^{\xi_{m}}-dA_{t}^{\xi_{n}})]\\
&  \geq \mathbb{E}_{\overline{G}_{\varepsilon}}\mathbb{[}\int_{0}^{T}%
\frac{\varepsilon|\hat{\eta}_{s}^{m,n}|^{2}}{\delta+|\hat{\eta}_{s}^{m,n}%
|}ds]-\mathbb{E}_{\overline{G}_{\varepsilon}}\mathbb{[}\int_{0}^{T}\mu
_{s}^{m,n}d(A_{t}^{\xi_{m}}-dA_{t}^{\xi_{n}})].
\end{align*}
Thus%
\begin{align*}
\mathbb{E}_{\overline{G}_{\varepsilon}}\mathbb{[}\int_{0}^{T}\frac
{\varepsilon|\hat{\eta}_{s}^{m,n}|^{2}}{\delta+|\hat{\eta}_{s}^{m,n}|}ds] &
\leq \mathbb{E}_{\overline{G}_{\varepsilon}}\mathbb{[}\int_{0}^{T}\mu_{s}%
^{m,n}d(A_{t}^{\xi_{m}}-dA_{t}^{\xi_{n}})]\\
&  \leq \rho(\xi_{m},\xi_{n}).
\end{align*}
It then follows that $\{ \eta^{n}\}_{n=1}^{\infty}$ is a Cauchy sequence in
$M_{\overline{G}_{\varepsilon}}^{1}(0,T)$. We then can pass limit on the both
sides of (\ref{Xi-m}) to obtain (\ref{Xi}). By using the same argument as Song
\cite{Song} Theorem 4.5, we can also prove that $\{z^{n}\}_{n=1}^{\infty}$ is
a Cauchy sequence in $H_{\overline{G}_{\varepsilon}}^{\alpha}(0,T)$. The proof
is complete.
\end{proof}

\end{document}